\documentclass[12pt,draft]{amsart}

\usepackage[ansinew]{inputenc} 
\usepackage[T1]{fontenc}
\usepackage{textcomp}
\usepackage{mathcomp}
\usepackage{amsmath,amsthm,amssymb,amsfonts,amsxtra}
\usepackage{mathtools}
\usepackage{paralist}
\usepackage{float}
\usepackage{comment}
\usepackage{enumerate}
\usepackage{tikz-cd}
\usepackage{mathrsfs}

\makeatletter
\@namedef{subjclassname@2010}{%
  \textup{2010} Mathematics Subject Classification}
\makeatother

\newtheorem{thm}{Theorem}[section]
\newtheorem{cor}[thm]{Corollary}
\newtheorem{prop}[thm]{Proposition}
\newtheorem{lem}[thm]{Lemma}

\theoremstyle{definition}

\theoremstyle{remark}

\newcommand{\set}[2]{\left \{ \: #1 \: \middle | \: #2 \: \right \} }
\newcommand{\join}[2]{J\left ( \: #1 \: \middle | \: #2 \: \right ) }

\newcommand{\exx}[1]{e^{i2\pi #1 }}
\newcommand{\wh}[1]{\widehat{#1}}
\newcommand{\wt}[1]{\widetilde{#1}}
\newcommand{\ov}[1]{\overline{#1}}

\newcommand{\conv}[1]{\stackrel{#1}{\longrightarrow}}

\newcommand{\norm}[1]{\lVert {#1} \rVert}

\DeclareMathOperator{\AP}{\mathcal{AP}}

\DeclareMathOperator{\LC}{\mathcal{LC}}
\DeclareMathOperator{\WAP}{\mathcal{WAP}}

\DeclareMathOperator{\ext}{\mathrm{ext} \,}

\DeclareMathOperator{\N}{\mathbb{N}}
\DeclareMathOperator{\Z}{\mathbb{Z}}
\DeclareMathOperator{\R}{\mathbb{R}}
\DeclareMathOperator{\Q}{\mathbb{Q}}
\DeclareMathOperator{\C}{\mathbb{C}}
\DeclareMathOperator{\T}{\mathbb{T}}

\DeclareMathOperator{\cA}{\mathcal{A}}
\DeclareMathOperator{\cB}{\mathcal{B}}
\DeclareMathOperator{\cC}{\mathcal{C}}
\DeclareMathOperator{\cD}{\mathcal{D}}

\DeclareMathOperator{\cF}{\mathcal{F}}
\DeclareMathOperator{\cG}{\mathcal{G}}
\DeclareMathOperator{\cH}{\mathcal{H}}

\DeclareMathOperator{\cM}{\mathcal{M}}

\DeclareMathOperator{\cQ}{\mathcal{Q}}

\DeclareMathOperator{\cU}{\mathcal{U}}

\DeclareMathOperator{\cX}{\mathcal{X}}

\DeclareMathOperator{\cnt}{\mathfrak{c}}

\begin{document}

\baselineskip=17pt

\title[Multiple disjointness and invariant measures]{Multiple disjointness and invariant measures on minimal distal flows}

\author[J. Rautio]{Juho Rautio}
\address{University of Oulu\\
   Department of Mathematical Sciences\\
   PL 8000\\
   FI-90014 Oulun yliopisto\\
   Finland}
\email{juho.rautio@oulu.fi}

\date{}

\begin{abstract}

As the main theorem, it is proved that a collection of minimal $PI$-flows with a common phase group and satisfying a certain algebraic condition is multiply disjoint if and only if the collection of the associated maximal equicontinuous factors is multiply disjoint. In particular, this result holds for collections of minimal distal flows. The disjointness techniques are combined with Furstenberg's example of a minimal distal system with multiple invariant measures to find the exact cardinalities of (extreme) invariant means on $\cD(\Z)$ and $\cD(\R)$, the spaces of distal functions on $\Z$ and $\R$, respectively. In all cases, this cardinality is $2^{\cnt}$. The size of the quotient of $\cD(\Z)$ or of $\cD(\R)$ by a closed subspace with a unique invariant mean is observed to be non-separable by applying the same ideas.

\end{abstract}



\maketitle

\section{Introduction}

The uniqueness of the normalised Haar measure on a compact Hausdorff topological group implies that there is a unique invariant mean on $\AP(T)$, the space of almost periodic functions on a topological group $T$. This invariant mean corresponds to the normalised Haar measure on the Bohr compactification $T^{\AP}$ of $T$, i.e., the universal topological group compactification. The space $\WAP(T)$ of weakly almost periodic functions on $T$ has likewise only one invariant mean or, equivalently, the universal semitopological semigroup compactification $T^{\WAP}$ of $T$ has a unique invariant probability measure, and again it is essentially the Haar measure on $T^{\AP}$, which we can see as the unique minimal ideal in $T^{\WAP}$ (see Theorems 2.14 and 3.12 in chapter 4 of \cite{bjm}, or Corollary 2.5 and Theorem 2.26 in \cite{burckel}). At the other end of the spectrum, if $T$ is an infinite, discrete, amenable group, then there are $2^{2^{|T|}}$ invariant means on $l^{\infty}(T)$ where $|T|$ is the cardinality of $T$ (\cite{chou76}). More generally, if $T$ is an amenable, locally compact, non-compact group and if $d$ denotes the smallest possible cardinality of a covering of $T$ by compact sets, then the space $\LC(T)$ of left norm continuous functions on $T$ has $2^{2^d}$ left invariant means (this is essentially proved in \cite{lp86}). The arguments involve counting the minimal left ideals in $T^{\LC}$, the compactification associated with $\LC(T)$. Recall that $T^{\LC}$ is the universal semigroup compactification of $T$ in the locally compact case (see \cite{bjm}, Theorem 5.7 on page 173). This line of research has been pursued further by Filali, Pym and Salmi (\cite{filpym03, filsalmi07}), and analogous results have been obtained in a Fourier algebra setting by Filali, Neufang and Monfared (\cite{fnm10}).

The work at hand was motivated by the problem of determining the cardinality of (left) invariant means on $\cD(T)$, the space of distal functions on $T$. This space can be used to construct the universal right topological group compactification of $T$ (\cite{bjm}, Theorem 6.5 on page 179), and it contains $\AP(T)$ but not necessarily $\WAP(T)$. The question can be formulated also as follows: what is the cardinality of invariant measures on the universal minimal distal flow with $T$ as the phase group? We address this problem in the cases $T = \Z$ and $T = \R$, and we shall show that the cardinality is as large as could be expected, namely $2^{\cnt}$ for both groups (Theorems~\ref{sec:cardz} and \ref{sec:cardr}). To be precise, we count the ergodic invariant measures on $T^{\cD}$, which are in essence the extreme invariant means on $\cD(T)$. As a by-product, we see that each minimal left ideal in $\beta \Z$ or in $\R^{\LC}$ supports $2^{\cnt}$ invariant measures (Corollary~\ref{sec:mlidc}).

For orientation, observe first that there must be more than one invariant measure on $\Z^{\cD}$ by Furstenberg's construction of a minimal, distal, non-uniquely ergodic system (\cite{furstenberg61}) since any invariant measure on a minimal distal system can be lifted to an invariant measure on the universal system $\Z^{\cD}$. By convexity, the cardinality we are after is at least $\cnt$, and on the other hand, the cardinality of invariant measures on the universal point-transitive system, i.e., the Stone-\v{C}ech compactification $\beta \Z$, sets an upper bound of $2^{\cnt}$. The latter can be obtained by counting the (mutually disjoint) minimal left ideals of $\beta \Z$, each of which supports at least one invariant measure (\cite{lp86}). This method is unsuitable in the distal case because $\Z^{\cD}$ is a group. In order to find $2^{\cnt}$ invariant measures on $\Z^{\cD}$, it is necessary to construct a single minimal distal system with this cardinality of invariant measures. Note that Furstenberg's example provides only a continuum of invariant measures, as it is defined on a metric space. However, taking an uncountable product of systems of this type leads to the desired conclusion: on each constituent system of the product, we can choose an invariant measure independently of the others, thus obtaining $2^{\cnt}$ distinct invariant product measures. The only problem that remains is ensuring that the product system is minimal, and to this end we need to study `multiple disjointness', generalising the usual notion of disjointness of two minimal flows to arbitrary collections. This is carried out in section \ref{sec:disjoint}, and the results we obtain are perhaps interesting in themselves. The main theorem is \ref{sec:main}, according to which the product of a collection of well-behaved minimal flows such as distal flows is minimal if and only if the product of the corresponding maximal equicontinuous factors is minimal. This is derived from a similar result in the context of two flows in \cite{egs76}. We also touch upon multiple disjointness for minimal equicontinuous flows with abelian phase group (Theorem~\ref{sec:indperp}), and we show how to construct an uncountable collection of minimal metric $PI$-flows with phase group $\R$ for which the product flow is minimal (Theorem~\ref{sec:perpreparam}).

The last section covers the arguments sketched above in detail for $\Z$ and also for $\R$. As a related phenomenon, the disjointness techniques and non-uniquely ergodic constructions are applied to show that, for any closed subspace $V$ of $\cD(T)$ with a single invariant mean, at least when $T = \Z$ or $T = \R$, the quotient $\cD(T)/V$ is non-separable. Results of this type have been obtained for other pairs of function spaces frequently encountered in abstract harmonic analysis (see for example the papers of Chou (\cite{chou82}) and Bouziad and Filali (\cite{boufil11})). The recent work of Filali and Galindo (\cite{filgal13}) contains a historical overview of the research in this area.

\section{Preliminaries}

The reader is assumed to be familiar with semigroup compactifications and the associated $m$-admissible function algebras. We follow \cite{bjm} in notations and terminology regarding this topic. In addition, the reader should be acquainted with topological dynamics, especially the algebraic aspects of flows on compact spaces. For background material on this subject, see \cite{auslander}, \cite{devries} and, to a lesser extent, \cite{bjm}. We recall the essentials concepts as well as some of the more specialised aspects of topological dynamics.

All topological spaces are assumed to be Hausdorff. The $C^{\ast}$-algebra of bounded, complex-valued functions on a set $X$ with supremum norm is $\cB(X)$. If $X$ is a topological space, then $\cC(X) \subseteq \cB(X)$ is the subspace of continuous functions. If $\pi \colon X \to Y$ is a mapping between two sets, we define an adjoint mapping $\pi^{\ast} \colon \cB(Y) \to \cB(X)$ by $\pi^{\ast}f = f \circ \pi$, $f \in \cB(Y)$. The weak$^{\ast}$ compact, convex set of all means on a subspace $\cF \subseteq \cB(X)$ is $M(\cF)$, and if $\cF$ is an algebra, then the set of multiplicative means on $\cF$ is denoted by $MM(\cF)$. The Stone-\v{C}ech compactification of a topological space $X$ is then $\beta X = MM(\cC(X))$, paired with the \emph{evaluation} $\epsilon \colon X \to \beta X$, that is, for $x \in X$, we define $\epsilon(x) \colon \cC(X) \to \C$ by $\epsilon(x)(f) = f(x)$, $f \in \cC(X)$.

When $S$ is a semigroup, the \emph{left} and \emph{right translations} by $s \in S$ are denoted by $\lambda_s$ and $\rho_s$. We define corresponding operators on $\cB(S)$ by $L_s = \lambda_s^{\ast}$ and $R_s = \rho_s^{\ast}$. If $S$ is equipped with a topology, it is \emph{right topological} if all right translations are continuous. A semigroup compactification of a topological group $T$ is a pair $(\phi,X)$ where $X$ is a right topological semigroup and $\phi \colon T \to X$ is a continuous homomorphism such that $\phi(T)$ is dense in $X$ and contained in the \emph{topological centre} $\Lambda(X)$ of all elements $x \in X$ for which $\lambda_x \colon X \to X$ is continuous. All semigroup compactifications of $T$ can be realised as pairs $(\epsilon,MM(\cF))$ where $\cF \subseteq \cC(X)$ is an $m$-admissible subalgebra, meaning that $\cF$ is a left and right translation invariant $C^{\ast}$-subalgebra containing the constants such that $T_\mu f(t) = \mu(L_t f)$, $t \in T$, defines a member of $\cF$ for any $\mu \in MM(\cF)$ and $f \in \cF$. The left introversion operators $T_\mu \colon \cF \to \cF$, $\mu \in MM(\cF)$, define a right topological semigroup structure by $\mu \nu = \mu \circ T_\nu$, $\mu, \nu \in MM(\cF)$.

For a topological group $T$, the space of left norm continuous functions on $T$ is $\LC(T)$, which is also the $m$-admissible subalgebra of $\cC(T)$ consisting of all functions that are uniformly continuous with respect to the right uniform structure on $T$. The corresponding compactification of $T$ is $T^{\LC} = MM(\LC(T))$. It is universal with respect to the joint continuity property, i.e., if $(\phi,X)$ is any semigroup compactification of $T$ such that the mapping $(t,x) \mapsto \phi(t)x : T \times X \to X$ is jointly continuous, then $(\phi,X)$ is a factor of $(\epsilon,T^{\LC})$ or, equivalently, $\phi^{\ast}\cC(X) \subseteq \LC(T)$. We shall usually refer to compactifications without the homomorphisms.

By a \emph{flow} or a $T$\emph{-flow} we mean a triple $(T,X,\alpha)$ consisting of a topological group $T$ (\emph{phase group}), a compact space (\emph{phase space}) and a continuous action $\alpha \colon T \times X \to X$. The action $\alpha$ is usually omitted from notation (so $\alpha(t,x) = tx$ and $(T,X,\alpha) = (T,X)$), and we shall often refer to the flow by its phase space alone. A \emph{dynamical system} or simply a \emph{system} is a $\Z$-flow written as a pair $(X,T)$ where $X$ is the phase space and $T \colon X \to X$ is the homeomorphism $Tx = 1x$, $x \in X$. We say that a flow $(T,Y)$ is a \emph{subflow} of $(T,X)$ if $Y \subseteq X$ is a non-empty, closed, $T$-invariant set. If $\{X_i\}_{i \in I}$ is a collection of $T$-flows, then the \emph{product flow} $X = \prod_{i \in I} X_i$ is a $T$-flow with the action defined by $(tx)_i = tx_i$, $t \in T$, $x \in X$, $i \in I$. 

A continuous, surjective, $T$-equivariant mapping $\pi \colon X \to Y$ between $T$-flows is a \emph{homomorphism}, and then $Y$ is said to be a \emph{factor} of $X$. Such a mapping $\pi$ induces a relation $R(\pi)$ on $X \times X$ defined as the set of all pairs $(x,x^{\prime})$ satisfying $\pi(x) = \pi(x^{\prime})$. It is a \emph{factor relation}, i.e., a closed equivalence relation that is invariant as a subset of the product flow $(T,X \times X)$. All factor relations on $X$ arise in this manner; if $R$ is a factor relation, then the quotient space $X/R$ is a compact Hausdorff space on which $T$ acts in a natural way so that the quotient map becomes a homomorphism. An \emph{isomorphism} is an injective homomorphism.

A point $x \in X$ in a $T$-flow is \emph{transitive} if its \emph{orbit} $Tx$ is dense. The flow $X$ is \emph{minimal} if it does not contain proper subflows or, equivalently, if all points are transitive. An ambit or a $T$-ambit is a pair $(X,x)$ where $X$ is a $T$-flow and the \emph{base point} $x \in X$ is transitive. A $T$-ambit $(X,x)$ can be represented as left translation invariant $C^{\ast}$-subalgebra of $\LC(T)$ that contains the constants; the operator $\phi_x \colon \cC(X) \to \LC(T)$, defined by $\phi_x f(t) = f(tx)$ for $f \in \cC(X)$ and $t \in T$, is an isometric $\ast$-homomorphism, and $\phi_x \cC(X)$ characterises $(X,x)$ up to ambit isomorphism (an isomorphism of flows that maps base point to base point). Note that all semigroup compactifications of a topological group $T$ with the joint continuity property can be regarded as ambits; the most natural choice for a base point is the identity element. Another common way to construct $T$-ambits is to select a function $f \in \LC(T)$ and to define the phase space as $X_f = \set{T_\mu f}{\mu \in T^{\LC}}$ with the topology of pointwise convergence, so $T$ acts on $X_f$ by right translations, and the base point is $f$. See \cite{devries}, section 5 in chapter 4 for more information on these matters (note also that $\LC(T)$ is denoted by $RUC^{\ast}(T)$ in this source).

The \emph{enveloping semigroup} of a flow $(T,X)$ is the right topological semigroup $E(T,X)$ or just $E(X)$ defined as the closure in $X^X$ of the set of all mappings of the form $x \mapsto tx : X \to X$, $t \in T$, equipped with the topology of pointwise convergence and composition as the semigroup operation. We may regard $E(X)$ as a semigroup compactification of $T$ with the joint continuity property. It acts on $X$ in a natural way, although this action is not jointly continuous, in general. We can also define an action of $T^{\LC}$ on $X$ via the canonical homomorphism from $T^{\LC}$ to $E(X)$.

Let $X$ be a $T$-flow. The \emph{proximal relation} $P(X)$ is defined as the set of all pairs $(x,x^{\prime}) \in X \times X$ for which $ax = ax^{\prime}$ for some $a \in E(X)$. These pairs are called \emph{proximal pairs}, and non-proximal pairs are \emph{distal pairs}. A \emph{distal point} is a point $x \in X$ such that $(x,x^{\prime})$ is a distal pair for all $x^{\prime} \in X \setminus \{x\}$. The flow $X$ is \emph{distal} if all points are distal, i.e., if $P(X) = \Delta_X$. A well-known characterisation of distality states that $X$ is distal if and only if $E(X)$ is a group. A function $f \in \LC(T)$ is \emph{distal} if the flow $X_f$ is distal, and $\cD(T)$ denotes the $m$-admissible algebra of all distal functions on $T$. The compactification $T^{\cD} = MM(\cD(T))$ is the universal (right topological) group compactification of $T$, and it is also the universal minimal distal $T$-flow in the sense that all other minimal distal $T$-flows are factors of $T^{\cD}$. A flow $(T,X)$ is \emph{equicontinuous} if $E(X)$ is an equicontinuous family of functions. This condition is equivalent to $E(X)$ being a topological group of homeomorphisms, so equicontinuity implies distality. The notions of proximality, distality and equicontinuity can also be defined for homomorphisms of flows, and the domains (as flows) of such homomorphisms are called \emph{proximal, distal} and \emph{equicontinuous extensions} of the co-domains, respectively. See \cite{auslander} or \cite{devries} for the definitions.

Let $T$ be a topological group, and let $u \in T^{\LC}$ be a minimal idempotent, so $G = u T^{\LC} u$ is algebraically a group with identity $u$. We say that a $T$-ambit $(X,x)$ is a $u$\emph{-ambit} if $ux = x$. The universal $u$-ambit is $(T^{\LC}u,u)$, which is a minimal flow since $T^{\LC}u$ is a minimal left ideal in $T^{\LC}$, so $u$-ambits are minimal. Conversely, any minimal $T$-flow $X$ can be turned into a $u$-ambit by picking a base point from the set $uX$. The \emph{structure group} of a $u$-ambit $(X,x)$ is the subgroup $\cG(X,x)$ of all $g \in G$ with $gx = x$. The group $G$ carries a topology called the $\tau$-topology, which is weaker than the relative topology inherited from $T^{\LC}$, and it is defined with respect to $u$ in such a way that $G$ becomes compact and $T_1$ with separately continuous group operation and continuous inversion. The structure groups of $u$-ambits are $\tau$-closed. See \cite{auslander} or \cite{devries} for more details.

A flow is said to be \emph{strictly} $PI$ if it is minimal and if there exists a transfinite sequence of factors of $X$, starting with the trivial flow, such that each successor is either an equicontinuous or a proximal extension of the previous one, and limit ordinals correspond to inverse limits of flows. A flow is $PI$ if it is a factor of a strictly $PI$-flow via a proximal homomorphism. All minimal distal flows are strictly $PI$ by the famous Furstenberg structure theorem, and point-distal minimal flows, i.e., those with a distal transitive point, are $PI$ (see \cite{devries}, Corollary 4.49 on page 577). An algebraic characterisation states that, if $(X,x)$ is a $u$-ambit for some minimal idempotent $u \in T^{\LC}$, then it is $PI$ if and only if the structure group $\cG(X,x)$ contains a certain subgroup $G_\infty$ of $G$ (\cite{egs75}, also \cite{auslander}, Theorem 23 on page 217). In the structure theory of minimal flows, in particular in the cited works, the phase group is often assumed to be discrete, so $T^{\LC} = \beta T$. But the algebraic characterisation of $PI$-flows holds in the more general topological case as well; all the arguments and constructions are analogous. The convention of using a topological phase group is followed in \cite{devries}.

Consider a flow $(T,X)$, and let $\cM(X)$ be the weakly compact, convex space of all regular probability measures defined on the Borel sets of $X$, so $\cM(X)$ can be identified with $M(\cC(X))$. We say that $\mu \in \cM(X)$ is an \emph{invariant measure} if $\mu(tA) = \mu(A)$ for all $t \in T$ and all Borel sets $A \subseteq X$. Let $\cM(T,X)$ denote the (possibly empty) closed, convex set of all invariant measures on $X$. We say that $\mu \in \cM(T,X)$ is ergodic if the invariant Borel sets $A \subseteq X$ satisfy $\mu(A) = 1$ or $\mu(A) = 0$. When $T$ is locally compact and second countable, the ergodic measures coincide with the extreme points of $\cM(T,X)$ (see \cite{bekkamayer}, Proposition 3.1). In general, the extreme points of $\cM(T,X)$ are ergodic. The invariant measures on $T^{\LC}$ correspond to the \emph{left invariant means} on $\LC(T)$, i.e., if $\mu \in \cM(T^{\LC})$, then $\mu$ is invariant if and only if the corresponding mean on $\LC(T)$, which we still denote by $\mu$, satisfies $\mu \circ L_t = \mu$ for all $t \in T$. Here we have also identified $\LC(T)$ with $\cC(T^{\LC})$ in a canonical way. We say that $(T,X)$ is \emph{uniquely ergodic} if $\cM(T,X)$ is a singleton, in which case the unique invariant measure is ergodic.

If $T$ is a locally compact group, it is amenable if the set $LIM(\LC(T))$ of left invariant means on $\LC(T)$ is nonempty or, equivalently, if all $T$-flows admit an invariant measure. All locally compact abelian groups are amenable. If $X$ and $Y$ are $T$-flows for a locally compact, amenable group $T$ and if $\pi \colon X \to Y$ is a homomorphism, then any invariant measure on $Y$ can be lifted to an invariant measure on $X$, that is, if $\nu \in \cM(T,Y)$, then there is some $\mu \in \cM(T,X)$ such that $\pi_{\ast} \mu = \nu$ where $\pi_{\ast} \colon \cM(X) \to \cM(Y)$ is the pushforward operator.

Suppose that $T$ is locally compact and amenable. A function $f \in \LC(T)$ is \emph{(left) almost convergent} to $c \in \C$ if $\mu(f) = c$ for all $\mu \in LIM(\LC(T))$. A $T$-ambit $(X,x)$ is uniquely ergodic if and only if the functions in $\phi_x \cC(X)$ are all left almost convergent.

\section{Maximal equicontinuous factors and disjointness}
\label{sec:disjoint}

Consider two minimal flows $X$ and $Y$ with a common phase group. Recall that they are said to be \emph{disjoint} if the product flow $X \times Y$ is minimal, and this is denoted by $X \perp Y$. If $\cX = \{X_i\}_{i \in I}$ is a collection of minimal flows with a common phase group, we say that $\cX$ is \emph{multiply disjoint} if the product flow $\prod_{i \in I} X_i$ is minimal, and we denote this by $\perp \cX$. Note that multiple disjointness of a collection of flows is equivalent to multiple disjointness of all finite subcollections due to the nature of the product topology. Chapter 11 of \cite{auslander} provides a concise treatment on the subject of disjointness, and the notion of multiple disjointness is taken from an exercise at the end of it. The purpose of this section is to find conditions that imply multiple disjointness. The main result we shall obtain is Theorem~\ref{sec:main}, which states that, for a collection of suitably `nice' minimal flows, multiple disjointness follows from the multiple disjointness of their maximal equicontinuous factors.

Recall that any flow $(T,X)$ has a \emph{maximal equicontinuous factor} $X^{eq}$, an equicontinuous factor of $X$ such that all other equicontinuous factors of $X$ are also factors of $X^{eq}$. If $\pi \colon X \to X^{eq}$ is a homomorphism, then $EQ(X) = R(\pi)$ is the \emph{equicontinuous structure relation} on $X$. It is the intersection of all factor relations on $X$ that induce an equicontinuous factor. An alternative way of defining $EQ(X)$ is based on the \emph{regionally proximal relation} $Q(X)$: if $\cU_X$ denotes the base for the uniform structure on $X$ consisting of neighbourhoods of the diagonal $\Delta_X \subseteq X \times X$, then $Q(X)$ is the intersection
	\[Q(X) = \bigcap_{U \in \cU_X} \ov{TU}.
\]
Here the elements $U \in \cU_X$ are treated as subsets of the product flow $X \times X$. Now, $EQ(X)$ is the smallest factor relation on $X$ that contains $Q(X)$. The latter is reflexive, symmetric, closed and invariant. Under certain dynamical conditions, $Q(X)$ is also transitive, so $Q(X) = EQ(X)$ (see \cite{devries}, remark 3 on page 400). Note that the flow $X$ also has a \emph{maximal distal factor} and a corresponding \emph{distal structure relation}, which is the smallest factor relation on $X$ that contains the proximal relation $P(X)$.

Disjointness of two minimal flows was studied on a very general level by Ellis, Glasner and Shapiro in \cite{egs76} in terms of algebras of functions defined on the phase group. In the cited paper, the phase group $T$ is discrete, so $T^{\LC} = \beta T$. One of the important subalgebras of $\LC(T) = l^{\infty}(T)$ used in \cite{egs76} is $\mathscr{K}$, which is defined with respect to a fixed minimal idempotent $u \in \beta T$ as follows:
	\[\mathscr{K} = \set{f \in l^{\infty}(T)}{T_u R_t T_u f = R_t f \text{ for all } t \in T}.
\]
Note that, for a $T$-ambit $(X,x)$, we have $\phi_x \cC(X) \subseteq \mathscr{K}$ if and only if $utux = tx$ for all $t \in T$. In this case, we say that $(X,x)$ is a $K(u)$\emph{-ambit} and that $x$ is a $K(u)$\emph{-point}. Such ambits have the nice property that $Q(X) = EQ(X)$ (\cite{ek71}). Any $K(u)$-ambit is a $u$-ambit. When $T$ is abelian, the two notions coincide. Also, if the base point $x$ of a $T$-ambit $(X,x)$ is distal, then $(X,x)$ is a $K(u)$-ambit. The well-known fact that $Q(X) = EQ(X)$ when $X$ is a minimal distal flow can be seen as a corollary of this observation.

The following theorem is translated into the language of ambits from the original presentation.

\begin{thm}
[{\cite[Theorem 4.2]{egs76}}]
\label{sec:egs}
Let $T$ be a discrete group, let $u \in \beta T$ be a minimal idempotent, and let $(X,x)$ and $(Y,y)$ be $u$-ambits with phase group $T$. Suppose that the following conditions hold:
\begin{itemize}
\item[(i)] $x$ or $y$ is a $K(u)$-point;
\item[(ii)] $G_\infty \subseteq \cG(X,x) \cG(Y,y)$;
\item[(iii)] $X^{eq} \perp Y^{eq}$.
\end{itemize}
Then, $X \perp Y$.
\end{thm}

The algebra $\mathscr{K}$ and the notion of $K(u)$-ambits can, of course, be defined for arbitrary topological groups $T$. We can always replace a topological phase group $T$ with its discretised version $T_d$, and any $v$-ambit or a $K(v)$-ambit for some minimal idempotent $v \in T^{\LC}$ can be regarded as a $u$-ambit or a $K(u)$-ambit, respectively, when we pick a minimal idempotent $u \in \beta T_d$ so that $\pi(u) = v$ for the canonical homomorphism $\pi \colon \beta T_d \to T^{\LC}$. In addition, it is not too difficult to prove the analogue of Theorem~\ref{sec:egs} for a general topological group $T$ as a corollary to the discrete version. Condition (ii) in the topological case implies an analogous statement in the discrete setting; the only part requiring some thought is the verification of the fact that the group $G_\infty$ in $T^{\LC}$ contains the $\pi$-image of its counterpart in $\beta T_d$. The argument involves a transfinite induction, and it also relies on the observation that the restriction $\pi$ to $u \beta T_d u$ is a closed, continuous group homomorphism onto $v T^{\LC} v$ with respect to the appropriate $\tau$-topologies.

The conditions of the theorem above are satisfied if $T$ is abelian and one of the flows is $PI$, or if one of the base points is distal. In order to find a similar condition for multiple disjointness, we must characterise the maximal equicontinuous factor of a product flow (Theorem~\ref{sec:prodeq}). Two simple lemmas are needed.

\begin{lem}
\label{sec:minrel}
Let $(T,X)$ be a flow, and let $R$ be a reflexive, symmetric, invariant relation on $X$. For each ordinal $\alpha$, define a relation $R_\alpha$ on $X$ by transfinite recursion as follows:
\begin{itemize}
\item[(1)] Define $R_0 = R$.
\item[(2)] If $R_\alpha$ has been defined for some ordinal $\alpha$, put $R_{\alpha+1} = R_\alpha \circ R_\alpha$.
\item[(3)] If $\beta$ is a limit ordinal and $R_\alpha$ has been defined for each $\alpha < \beta$, put $R_\beta = \ov{\bigcup_{\alpha < \beta} R_\alpha}$.
\end{itemize}
Then, there exists an ordinal $\theta$ for which $R_\alpha = R_\theta$ for all $\alpha \geq \theta$, and $R_\theta$ is the smallest factor relation on $X$ that contains $R$.
\end{lem}

\begin{proof}
Let $R^{\prime}$ be the smallest factor relation on $X$ that contains $R$. It is obtained as the intersection of the collection of all factor relations containing $R$, one of which is $X \times X$. A straightforward argument by transfinite induction shows that, for any ordinal $\alpha$, the relation $R_\alpha$ is reflexive, symmetric, invariant, and $R \subseteq R_\alpha \subseteq R^{\prime}$. The transfinite sequence is increasing and therefore eventually constant, say $R_\alpha = R_\theta$ for all $\alpha \geq \theta$. It follows that $R_\theta = R_\theta \circ R_\theta$, i.e., this relation is transitive. Also, $R_\theta$ must be closed by the limit ordinal step. In conclusion, $R_\theta$ is a factor relation with $R \subseteq R_\theta \subseteq R^{\prime}$. We infer that $R_\theta = R^{\prime}$.
\end{proof}

Suppose that $\{X_i\}_{i \in I}$ is a collection of sets, and suppose that $\cQ = \{Q_i\}_{i \in I}$ is a collection of relations such that $Q_i \subseteq X_i \times X_i$ for each $i \in I$. Put $X = \prod_{i \in I}$. We define their product as a relation on $X$ by
	\[\bigotimes \cQ = \bigotimes_{i \in I} Q_i = \set{(x,x^{\prime}) \in X \times X}{(x_i,x^{\prime}_i) \in Q_i \text{ for all } i \in I}.
\]
Products of equivalence relations are again equivalence relations. Similarly, the product of factor relations is a factor relation on the product flow.

\begin{lem}
\label{sec:rellem}
Let $\{X_i\}_{i \in I}$ be a collection of topological spaces, and let $Q_i$ be a relation on $X_i$ for each $i \in I$. Define $X = \prod_{i \in I} X_i$ with the product topology. For every set $J \subseteq I$, let $\cQ(J) = \{Q_i(J)\}_{i \in I}$ be the collection of relations defined by $Q_i(J) = Q_i$ when $i \in J$ and $Q_i(J) = \Delta_{X_i}$ otherwise. Let $R$ be a closed, transitive relation on $X$ containing each $\bigotimes \cQ(\{i\})$, $i \in I$. Then, $R$ also contains the relation $\bigotimes \cQ(I)$.
\end{lem}

\begin{proof}
The first step is to show that $\bigotimes \cQ(J) \subseteq R$ for all (non-empty) finite sets $J \subseteq I$, and this is done by induction on the cardinality of $J$. By assumption, the claim holds when $|J| = 1$. Suppose that $\bigotimes \cQ(K) \subseteq R$ for any set $K \subseteq I$ of cardinality $n \in \N$. Consider a set $J \subseteq I$ of cardinality $n+1$, so $J = K \cup \{j\}$ for some $K \subseteq I$ of cardinality $n$ and for some $j \in I \setminus K$. If $(x,x^{\prime}) \in \bigotimes \cQ(J)$, then for a suitably chosen $y \in X$, we get $(x,y) \in \bigotimes \cQ(\{j\})$ and $(y,x^{\prime}) \in \bigotimes \cQ(K)$, so $(x,x^{\prime}) \in R$ by the assumptions.

To prove the general case, consider an arbitrary pair $(x,x^{\prime}) \in \bigotimes \cQ(I)$. Let $F$ be the collection of non-empty finite subsets of $I$, ordered by inclusion. For each $J \in F$, define $x_J \in X$ so that $(x_J)_j = x_j$ for all $j \in J$ and $(x_J)_i = x^{\prime}_i$ for $i \in I \setminus J$. Then, $(x_J,x^{\prime}) \in \bigotimes \cQ(J) \subseteq R$ for every $J \in F$, and the net $(x_J)$ converges to $x$. Since $R$ is closed, we have $(x,x^{\prime}) \in R$.
\end{proof}

\begin{thm}
\label{sec:prodeq}
Let $\{X_i\}_{i \in I}$ be a collection of flows with a common phase group $T$, and let $X = \prod_{i \in I} X_i$. Then,
	\[EQ(X) = \bigotimes_{i \in I} EQ(X_i).
\]
In other words, the maximal equicontinuous factor $X^{eq}$ of $X$ is isomorphic to the product flow $\prod_{i \in I} X_i^{eq}$.
\end{thm}

\begin{proof}
The product flow $\prod_{i \in I} X_i^{eq}$ is equicontinuous and clearly isomorphic to $X/\bigotimes_{i \in I} EQ(X_i)$, so $EQ(X) \subseteq \bigotimes_{i \in I} EQ(X_i)$.

For all subsets $J \subseteq I$, define a collection $\cQ(J) = \{Q_i(J)\}_{i \in I}$ of relations by putting $Q_j(J) = Q(X_j)$ when $j \in J$ and $Q_i(J) = \Delta_{X_i}$ otherwise. Similarly, for $J \subseteq I$, define $\mathcal{EQ}(J) = \{EQ_i(J)\}_{i \in I}$ by $EQ_j(J) = EQ(X_j)$ for all $j \in J$ and $EQ_i(J) = \Delta_{X_i}$ otherwise. We claim that, for all $j \in I$, the relation $\bigotimes \mathcal{EQ}(\{j\})$ is the smallest factor relation on $X$ containing $\bigotimes \cQ(\{j\})$. This seems intuitive because of the characterisation of the equicontinuous structure relation as the smallest factor relation containing the regionally proximal relation (for any flow). We need Lemma~\ref{sec:minrel} for a rigorous argument. Fix $j \in I$. Define a transfinite sequence $(Q_\alpha(X_j))$ of relations on $X_j$, indexed by ordinals $\alpha$, in the manner of Lemma~\ref{sec:minrel}, starting with $Q_0(X_j) = Q(X_j)$. We know that this sequence ultimately reaches $EQ(X_j)$. For each ordinal $\alpha$, define a collection $\cQ_\alpha(\{j\}) = \{Q_{\alpha,i}(\{j\})\}_{i \in I}$ by $Q_{\alpha,j}(\{j\}) = Q_\alpha(X_j)$ and $Q_{\alpha,i}(\{j\}) = \Delta_{X_i}$ otherwise. Defining $R = \bigotimes \cQ_0(\{j\})$ and by applying the transfinite construction of Lemma~\ref{sec:minrel} to this relation, we obtain another transfinite sequence $(R_\alpha)$. It is not difficult to verify that $R_\alpha = \bigotimes \cQ_\alpha(\{j\})$ for each ordinal $\alpha$ by using transfinite induction. Therefore, the sequence $(R_\alpha)$ reaches the factor relation $\bigotimes \mathcal{EQ}(\{j\})$, from which the desired auxiliary claim follows.

Observe that $\bigotimes \cQ(\{i\}) \subseteq Q(X)$ for all $i \in I$ (use the net characterisation of the regionally proximal relation, see \cite{devries}, Q.3 on page 397). Therefore also $\bigotimes \cQ(\{i\}) \subseteq EQ(X)$ and $\bigotimes \mathcal{EQ}(\{i\}) \subseteq EQ(X)$ for all $i \in I$. Using Lemma~\ref{sec:rellem}, we infer that $\bigotimes_{i \in I} EQ(X_i) = \bigotimes \mathcal{EQ}(I) \subseteq EQ(X)$, as required.
\end{proof}

Similarly, the maximal distal factor of the product $\prod_{i \in I} X_i$ is isomorphic to the product of the respective maximal distal factors. The arguments are essentially the same as above with the exception that the regionally proximal relations are replaced with proximal relations.

Next, we focus on the problem of characterising multiple disjointness for minimal equicontinuous flows with an abelian phase group (Theorem~\ref{sec:indperp}).

Let $G$ be an abelian group with identity $e$, and let $\cH = \{H_i\}_{i \in I}$ be a collection of subgroups of $G$. We say that $\cH$ is \emph{independent} if, whenever $J \subseteq I$ is non-empty and finite and $\prod_{j \in J} h_j = e$ for some $h_j \in H_j$, $j \in J$, we must have $h_j = e$ for all $j \in J$. It is clear that a collection of subgroups is independent if and only if all non-empty, finite subcollections are independent.

Recall that, when $T$ is an abelian topological group, any minimal equicontinuous $T$-flow arises from a topological group compactification $(\phi,X)$ of $T$ (\cite{devries}, Corollary 3.42 on page 317); $T$ acts on $X$ by $tx = \phi(t)x$, $t \in T$, $x \in X$. We denote the character group of $T$ by $\wh{T}$.

\begin{thm}
\label{sec:indperp}
Let $T$ be an abelian topological group, and let $\{(\phi_i,X_i)\}_{i \in I}$ be a collection of topological group compactifications of $T$. For each $i \in I$, let $A_i = \phi_i^{\ast}(\wh{X_i})$. Then, the family $\{X_i\}_{i \in I}$ of $T$-flows is multiply disjoint if and only if the family $\{A_i\}_{i \in I}$ of subgroups of $\wh{T}$ is independent.
\end{thm} 

\begin{proof}
We may assume that $I$ is finite since the general case reduces to this one.

Suppose first that $\perp \{X_i\}_{i \in I}$. Let $X = \prod_{i \in I} X_i$, and let $\chi_i \in \wh{X_i}$, $i \in I$, be such that $\prod_{i \in I} \phi_i^{\ast} \chi_i = 1$. Define a continuous function $\chi \colon X \to \T$ by $\chi(x) = \prod_{i \in I} \chi_i(x_i)$ for all $x \in X$. Let $e_i$ be the identity of $X_i$ for each $i \in I$, so $e = (e_i)_{i \in I}$ is the identity of the product group $X$. Now, for any $t \in T$,
	\[\chi(te) = \prod_{i \in I} \chi_i(te_i) = \prod_{i \in I} \chi_i(\phi_i(t)) = 1.
\]
Since $X$ is minimal, $\chi = 1$. Thus, for any $i \in I$ and $x \in X$ with $x_j = e_j$ for $j \in I \setminus \{i\}$, we get $\chi_i(x_i) = \chi(x) = 1$. This shows that $\chi_i = 1$ for all $i \in I$. Consequently, $\phi_i^{\ast} \chi_i = 1$ for all $i \in I$, and the collection $\{A_i\}_{i \in I}$ is independent.

Suppose then that $\{X_i\}_{i \in I}$ is not multiply disjoint, so $X$ is not minimal. Let $\phi \colon T \to X$ be the continuous homomorphism $\phi(t)_i = \phi_i(t)$, $i \in I, t \in T$, and let $Y = \ov{\phi(T)}$, both a closed subgroup of $X$ and the orbit closure of the identity. Since $X$ is distal, all orbit closures are minimal sets, so $Y$ is a proper subset of $X$. Pick an arbitrary $z \in X \setminus Y$. We can find a character $\chi \in \wh{X}$ so that $\chi(y) = 1$ for all $y \in Y$ and $\chi(z) \neq 1$. It must be of the form $\chi(x) = \prod_{i \in I}\chi_i(x_i)$ for all $x \in X$ for some $\chi_i \in \wh{X_i}$, $i \in I$. Clearly, at least some of the characters $\chi_i$ must be non-trivial, so the corresponding characters $\phi_i^{\ast} \chi_i$ of $T$ are also non-trivial. On the other hand, $\prod_{i \in I} \phi_i^{\ast} \chi_i(t) = \chi(\phi(t)) = 1$ for every $t \in T$, so the collection $\{A_i\}_{i \in I}$ is not independent.
\end{proof}

This theorem can be used to obtain the well-known characterisations of multiple disjointness of finite collections of continuous rotations and discrete irrational rotations of the circle (\cite{devries}, 1.14 on page 157), and these results are easily generalised to infinite collections.

Finally, we can state the main theorem of this section. Note that we assume the phase group to be topological as opposed to discrete, so when we invoke Theorem~\ref{sec:egs}, we are actually referring to the version with topological $T$. But as we have noted, the two versions are equivalent.

\begin{thm}
\label{sec:main}
Let $T$ be a topological group, let $u \in T^{\LC}$ be a minimal idempotent, and let $\cX = \{(X_i,x_i)\}_{i \in I}$ be a collection of $K(u)$-ambits with phase group $T$ such that each $X_i$ is a $PI$-flow. Let $\cX^{eq} = \{X_i^{eq}\}_{i \in I}$, and suppose that $\perp \cX^{eq}$. Then, $\perp \cX$.
\end{thm}

\begin{proof}
Again, we prove the theorem for finite index sets $I$.

The case $|I| = 1$ is trivial. Suppose that the claim of the theorem holds whenever $|I| = n$ for some $n \in \N$. Consider a collection $\cX = \{(X_i,x_i)\}_{i \in I}$ of $K(u)$-ambits with phase group $T$ such that each is $PI$, $\perp \cX^{eq}$, and $|I| = n+1$. Pick $k \in I$, and define $J = I \setminus \{k\}$ and $X = \prod_{j \in J} X_j$. By assumption, $X$ is minimal, and its maximal equicontinuous factor $X^{eq}$ is isomorphic to $\prod_{j \in J} X_j^{eq}$ by Theorem~\ref{sec:prodeq}. Moreover, the point $x = (x_j)_{j \in J}$ is a $K(u)$-point, and $G_\infty \subseteq \bigcap_{j \in J} \cG(X_j,x_j) = \cG(X,x)$, so $X$ is $PI$. We can apply Theorem~\ref{sec:egs} to conclude that $X \perp X_k$, completing the argument.
\end{proof}

We shall now turn our attention to $\R$-flows to prove a specialised disjointness theorem. For the sake of clarity, we denote such a flow by $(X,\sigma)$, $X$ being the phase space and $\sigma \colon \R \times X \to X$ the action. The maximal equicontinuous factor of an $\R$-flow $F = (X,\sigma)$ is denoted by $F^{eq}$ instead of just $X^{eq}$. For any $\R$-flow $F = (X,\sigma)$ and for any $a \in \R$, $a > 0$, we define a new $\R$-flow $F_a = (X,a\sigma)$ by keeping the same phase space and defining a new action $a \sigma(t,x) = \sigma(at,x)$ for all $t \in \R$ and $x \in X$. This manipulation retains all the essential dynamical features such as orbits and invariant measures, as it simply alters the `speed' by which the phase group acts on $X$. Hence, $F_a$ is distal, equicontinuous, minimal or uniquely ergodic for some $a > 0$ if and only if $F$ has the same property. If $F^{\prime} = (X^{\prime},\sigma^{\prime})$ is another $\R$-flow and if $\pi \colon X \to X^{\prime}$ is a homomorphism, then it is also a homomorphism from $F_a$ to $G_a$ for any $a > 0$, and the dynamical properties of $\pi$ such as equicontinuity and proximality are not affected by the change of actions. Consequently, it is easy to see that $F_a$ is $PI$ for any $a > 0$ if $F$ is $PI$. Another point worth noting is that the enveloping semigroups of $F = (X,\sigma)$ and $F_a$ are identical for any $a > 0$. Also, the regionally proximal relations $Q(F)$ and $Q(F_a)$ are identical subsets of $X \times X$. It follows that $EQ(F) = EQ(F_a)$, and we may write $(F_a)^{eq} = (F^{eq})_a = F_a^{eq}$ with no ambiguity.

We can now build a large multiply disjoint collection of metric, minimal, $PI$ $\R$-flows out of a single one:

\begin{thm}
\label{sec:perpreparam}
Let $F = (X,\sigma)$ be a metric, minimal, $PI$ $\R$-flow. Then, there exists an uncountable set $A \subseteq \R$ such that $\perp \{F_a\}_{a \in A}$.
\end{thm}

\begin{proof}
Put $E = E(F^{eq}) = E(F_a^{eq})$, $a > 0$, so $E$ is a topological group. Let $\psi_a \colon \R \to E$ denote the canonical continuous homomorphism from $\R$ into the enveloping semigroup of $F_a^{eq}$ for any $a > 0$, that is, $\psi_a(t) = \psi_1(at)$ for all $t \in \R$. Each $F_a$ can be turned into a $K(u)$-ambit for a fixed minimal idempotent $u \in \R^{\LC}$ by picking a base point $x_a$ from the phase space of $F_a$ such that $ux_a = x_a$. By Theorem~\ref{sec:main}, it suffices to find an uncountable set $A \subseteq \R$ such that $\perp \{F_a^{eq}\}_{a \in A}$. Since the flow on the enveloping semigroup of $F_a^{eq}$ is isomorphic to $F_a^{eq}$ for any $a > 0$ due to $\R$ being abelian, a collection $\{F_a^{eq}\}_{a \in A}$ for some non-empty $A \subseteq \R$ is multiply disjoint if and only the corresponding collection $\{E(F_a^{eq})\}_{a \in A}$ is multiply disjoint. By Theorem~\ref{sec:indperp}, the latter is equivalent to the independence of the subgroups $\psi_a^{\ast}(\wh{E})$ of $\wh{\R}$, $a \in A$. 

Since $X$ is metrisable, so is $X^{eq}$ and therefore also $E$, which is isomorphic to $F^{eq}$ as a flow. It follows that the character group $\wh{E}$ is countable (\cite{hofmann}, Theorem 8.45). Let $\Phi \colon \wh{\R} \to \R$ be the inverse of the topological isomorphism that maps $r \in \R$ to the character $x \mapsto e^{irx} : \R \to \T$. We define $G_a = \Phi(\psi_a^{\ast}(\wh{E}))$ for all $a > 0$, and we put $G = G_1$. Now, $G_a = aG$ for any $a > 0$. The group $G$ is countable. It remains to find an uncountable $A \subseteq (0,\infty)$ so that $\{aG\}_{a \in A}$ is independent.

Let $\cA$ be the (non-empty) family of all non-empty subsets of $(0,\infty)$ for which $\{aG\}_{a \in A}$ is independent. Inclusion provides a partial order on $\cA$. If $\{C_i\}_{i \in I} \subseteq \cA$ is a chain in $\cA$, put $C = \bigcup_{i \in I} C_i$. This is an upper bound for $\{C_i\}_{i \in I}$, and $\{cG\}_{c \in C}$ is easily seen to be independent. By Zorn's lemma, there is a maximal element $A \in \cA$. This set must be uncountable, which we prove by an argument by contradiction. If $A$ is countable, so are $AG$ and the subgroup $\Gamma \subseteq \R$ generated by $AG$. We define $B \subseteq \R$ as the set of all elements of the form $b = \gamma/g$ where $\gamma \in \Gamma$ and $g \in G \setminus \{0\}$. Again, $B$ is countable, so we can pick $a \in (0,\infty) \setminus B$. The set $\{a\} \cup A$ is now in $\cA$. To see this, suppose that $a_i \in A$ and $g,g_i \in G$ for $1 \leq i \leq n$, $n \in \N$, are such that $ag + \sum_{i = 1}^n a_i g_i = 0$. If $g = 0$, we also get $g_i = 0$ for all $1 \leq i \leq n$ from the independence of $\{a_iG\}_{i = 1}^n$. If $g \neq 0$, we get $a = \gamma/g \in B$ where $\gamma = -\sum_{i = 1}^n a_i g_i \in \Gamma$, contradicting the choice of $a$, so this case is not possible. Thus, we must have $ag = a_i g_i = 0$ for $1 \leq i \leq n$, showing that $\{a\} \cup A \in \cA$. But this contradicts the maximality of $A$. In conclusion, $A$ is uncountable.
\end{proof}

\section{Applications}

We can now use the multiple disjointness results to answer the original problem of finding the cardinalities of invariant means on the spaces $\cD(\Z)$ and $\cD(\R)$.

In \cite{furstenberg61}, Furstenberg gave an example of a minimal distal system on $\T^2$ with multiple invariant measures. It extends a particular irrational rotation of the circle, but as Kodaka pointed out in \cite{kodaka95}, an analogous system can be constructed for any irrational rotation angle. To be specific, given any $\alpha \in \R \setminus \Q$ and defining $a = \exx{\alpha} \in \T$, there is a continuous function $t_a \colon \T \to \T$ so that the homeomorphism $T_a \colon \T^2 \to \T^2$, $T_a(x,y) = (ax,t_a(x)y)$, $(x,y) \in \T^2$, defines a minimal distal system $(\T^2,T_a)$ that is not uniquely ergodic.

More generally, if the irrational rotation angle $\alpha$ is fixed and $t \colon \T \to \T$ is a continuous function, then we can consider a homeomorphism $T \colon \T^2 \to \T^2$ defined by $T(x,y) = (ax,t(x)y)$, $(x,y) \in \T^2$, and the minimality and unique ergodicity of the system $(\T^2,T)$ can be determined by considering the following functional equations:
\begin{equation}
\frac{f(ax)}{f(x)} = t(x)^m \text{ for all } x \in \T \label{eq:mineq}
\end{equation}
where $f \colon \T \to \T$ is a continuous function and $m \in \Z \setminus \{0\}$;
\begin{equation}
\frac{g(ax)}{g(x)} = t(x)^n \text{ for almost every } x \in \T  \label{eq:ueeq}
\end{equation}
where $g \colon \T \to \T$ is a Borel function and $n \in \Z \setminus \{0\}$. The space $\T$ is understood as a measure space with respect to its normalised Haar measure, which is the unique invariant measure for $(\T,\lambda_a)$. The system $(\T^2,T_a)$ is minimal if and only if \eqref{eq:mineq} has no solution $f$ and $m$, and the system is uniquely ergodic if and only if \eqref{eq:ueeq} has no solution $g$ and $n$. It can be shown that, if a solution $g$ exists for the second equation with $n = 1$ and if, for any $k \in \Z \setminus \{0\}$, the function $g^k$ does not agree almost everywhere with any continuous function, then the first equation has no solution $f$ and $m$. The function $t_a$ constructed in \cite{kodaka95} has precisely this property.

\begin{prop}
The maximal equicontinuous factor of $(\T^2,T_a)$ is $(\T,\lambda_a)$ for any $a = \exx{\alpha}$, $\alpha \in \R \setminus \Q$.
\end{prop}

\begin{proof}
Let $a = \exx{\alpha}$ for some $\alpha \in \R \setminus \Q$. We must show that the regionally proximal relation $Q = Q(\T^2,T_a)$ coincides with $R(\pi)$ where $\pi \colon \T^2 \to \T$ is the projection onto the first coordinate. Since $\pi$ is a homomorphism to an equicontinuous factor, we have $Q \subseteq R(\pi)$.

For any $g \in \T$, let $S_g \colon \T^2 \to \T^2$ be the mapping $S_g(x,y) = (x,yg)$, $(x,y) \in \T^2$. It is an automorphism of the system $(\T^2,T_a)$, i.e., an isomorphism from the system to itself. Thus, $S_g \times S_g(Q) = Q$ for any $g \in \T$. For each $w = (x,y) \in \T^2$, let $G(w) \subseteq \T$ be the set
	\[G(w) = \set{g \in \T}{(w,S_g(w)) \in Q}.
\]
It is easy to check that this is a closed subgroup of $\T$ for any $w \in \T^2$. We claim that $G(w)$ does not depend on the choice of $w$. This follows from the minimality of $(\T^2,T_a)$: given $w, w^{\prime} \in \T^2$, there exists some $p \in E(\T^2,T_a)$ so that $pw = w^{\prime}$, and if $g \in G(w)$, then
	\[(w^{\prime},S_g(w^{\prime})) = (pw,S_g(pw)) = p \times p(w,S_g(w)) \subseteq p \times p(Q) \subseteq Q.
\]
This proves that $G(w) \subseteq G(w^{\prime})$, and the reverse inclusion is proved analogously. Put $G = G(w)$ for any $w \in \T^2$.

We must show that $G = \T$. Knowing that $G$ is a closed subgroup of $\T$, we only need to show that $G$ cannot be a finite cyclic group. Suppose that $|G| = k \in \N$, so $G$ is generated by some $g \in \T$ with $g^k = 1$. Define $U \colon \T^2 \to \T^2$ by $U(x,y) = (ax,t_a(x)^k y)$, $(x,y) \in \T^2$, and define a homomorphism $\phi$ from $(\T^2,T_a)$ to $(\T^2,U)$ by $\phi(x,y) = (x,y^k)$, $(x,y) \in \T^2$. Now, $R(\phi) = Q$, so $(\T^2,U)$ is the maximal equicontinuous factor of $(\T^2,T_a)$, and $(\T^2,U)$ is uniquely ergodic. If $b \in B(\T,\T)$ is such that $b(ax)/b(x) = t_a(x)$ for almost every $x \in \T$, then $b(ax)^k/b(x)^k = t_a(x)^k$ for almost every $x \in \T$. In other words, equation~\eqref{eq:ueeq} for $(\T^2,U)$ is solved by $b^k \in B(\T,\T)$ and $n = 1$, so $(\T^2,U)$ is not uniquely ergodic, and we have arrived at the desired contradiction.
\end{proof}

Forming a large enough multiply disjoint collection of systems $(\T^2,T_a)$ leads us to the exact cardinality of the set of extreme invariant means on $\cD(\Z)$. But first, we must recall some technical details about joinings of invariant measures.

We can construct invariant measures on product flows as products of invariant measures. Firstly, consider a collection $\{X_i\}_{i \in I}$ of compact spaces, define $X_J = \prod_{j \in J} X_j$ for every non-empty subset $J \subseteq I$, and let $\pi_J \colon X_I \to X_J$ and $\pi_i \colon X_I \to X_i$, $i \in I$, be projections. If we pick $\mu_i \in \cM(X_i)$ for each $i \in I$, then there is a unique \emph{product measure} $\mu = \bigotimes_{i \in I} \mu_i \in \cM(X_I)$ such that, for any non-empty finite $J \subseteq I$ and for any Borel sets $A_j \subseteq X_j$, $j \in J$,
	\[\mu \left( \bigcap_{j \in J} \pi_j^{-1}(A_j) \right) = \prod_{j \in J} \mu_j(A_j).
\]
In particular, $(\pi_i)_{\ast} \mu = \mu_i$ for all $i \in I$. If $\{X_i\}_{i \in I}$ is a collection of flows that share a common phase group $T$ and if each $\mu_i$ is invariant, then the product measure $\mu$ is also invariant. More generally, a \emph{joining} of the measures $\mu_i$ is any $\mu \in \cM(T,X_I)$ such that $(\pi_i)_{\ast}\mu = \mu_i$ for every $i \in I$. The set $\join{\mu_i}{i \in I}$ of all such joinings is a (non-empty) compact, convex subspace of $\cM(T,X_I)$, and if each $\mu_i$ is an extreme point in $\cM(T,X_i)$, then all extreme points of $\join{\mu_i}{i \in I}$ are also extreme points in $\cM(T,X_I)$. Recall that when $T$ is locally compact and second countable, in particular if $T = \Z$ or $T = \R$, then the extreme points coincide with ergodic measures. 

\begin{thm}
\label{sec:cardz}
The cardinality of extreme invariant means on $\cD(\Z)$, which is the cardinality of ergodic measures on the universal minimal distal $\Z$-flow, is $2^{\cnt}$.
\end{thm}

\begin{proof}
For each $a = \exx{\alpha}$, $\alpha \in \R \setminus \Q$, let $(\T^2,T_a)$ be the minimal distal system from the beginning of this section, and pick distinct ergodic measures $\mu_{a,0}$ and $\mu_{a,1}$ from $\cM(\T^2,T_a)$. Let $B_1$ be a Hamel basis for the vector space $\R$ over the scalar field $\Q$, and assume that $1 \in B_1$. Define $B = B_1 \setminus \{1\}$, so $B$ consists of irrational numbers. Let $A \subseteq \T$ be the set of the elements $\exx{\alpha}$ for $\alpha \in B$. Note that the mapping $\alpha \mapsto \exx{\alpha} : B \to A$ is injective, so the cardinality of $A$ is the cardinality of $B$, namely $\cnt$. Now, the system $(\T,\lambda_a)$ is minimal for any $a \in A$, and the collection $\{(\T,\lambda_a)\}_{a \in A}$ is multiply disjoint. First, we use Theorem~\ref{sec:indperp} to show that the maximal equicontinuous factors form a multiply disjoint collection: associating $(\T,\lambda_a)$, $a \in A$, with the compactification $(\phi_a,\T)$ of $\Z$ where $\phi_a(n) = a^n$, $n \in \Z$, and identifying the subgroup $\phi_a^{\ast}(\wh{\T})$ of $\wh{\Z}$ with the subgroup $G_a = \{a^n\}_{n \in \Z} \subseteq \T$, the collection $\{G_a\}_{a \in A}$ is independent. Thus, the product system $(X,T)$, $X = (\T^2)^A$ and $T(x)_a = T_a(x_a)$ for all $x \in X$ and $a \in A$, is minimal by Theorem~\ref{sec:main}.

For any $\omega \in \{0,1\}^A$, we may consider the joining of the measures $\mu_{a,\omega(a)}$, $a \in A$, so $J(\omega) = \join{\mu_{a,\omega(a)}}{a \in A}$ is a non-empty subspace of $M(X,T)$, and each $\mu \in J(\omega)$ maps to the measures $\mu_{a,\omega(a)}$ under the pushforwards induced by the projections from $X$ to $X_a$. Clearly, $J(\omega) \cap J(\omega^{\prime}) = \emptyset$ when $\omega, \omega^{\prime} \in \{0,1\}^A$ are distinct. For each $\omega \in \{0,1\}^A$, pick $\mu_\omega \in \ext J(\omega)$, so $\mu_\omega$ is an ergodic measure on $X$. We see that the set $\set{\mu_\omega}{\omega \in \{0,1\}^A}$ has cardinality $2^{\cnt}$. Each $\mu_\omega$ can be lifted to an ergodic measure on the universal minimal distal system $\Z^{\cD}$, so the latter admits at least $2^{\cnt}$ ergodic measures. In other words, there are at least $2^{\cnt}$ extreme invariant means on $\cD(\Z)$. On the other hand, the universal point-transitive system, which is simply $\beta \Z$, has exactly $2^{\cnt}$ invariant measures since $l^{\infty}(\Z)$ has exactly $2^{\cnt}$ invariant means (\cite{chou76}), so the cardinality of ergodic measures on $\Z^{\cD}$ is also $2^{\cnt}$.
\end{proof}

It is apparent from the last argument that the cardinality of all invariant means on $\cD(\Z)$ is also $2^{\cnt}$.

As in the case of $\Z$-actions, we can find a minimal distal $\R$-flow with (at least) $2^{\cnt}$ ergodic measures. This is achieved by constructing a metric, minimal, distal $\R$-flow $F = (X,\sigma)$ admitting multiple ergodic measures (necessarily of cardinality $\cnt$ due to metrisability), applying Theorem~\ref{sec:perpreparam} to find an uncountable set $A \subseteq \R$ so that $\perp \{F_a\}_{a \in A}$, fixing distinct ergodic measures $\mu_{a,0}$ and $\mu_{a,1}$ for $F_a$, $a \in A$, and picking an extreme point from each of the $2^{\cnt}$ pairwise disjoint sets $\join{\mu_{a,\omega(a)}}{a \in A}$, $\omega \in \{0,1\}^A$. The rest of the argument is also similar to the proof of Theorem~\ref{sec:cardz}; the cardinality of invariant measures on $\LC(\R)$ is again $2^{\cnt}$ (\cite{lp86}). Finding $F$ is the only remaining task. The idea is to interpolate a suitable distal function on $\Z$ to a distal function on $\R$. The result quoted below shows how this is done.

\begin{thm}[{\cite[Theorem 4.8(iv)]{lmp92}}]
Let $G$ be a locally compact group, let $N$ be a closed normal subgroup of $G$, and suppose that there is a compact set $K \subseteq G$ such that $G = KN$. Let $g \in \cC(K \times N)$ be such that $g(h,m) = g(k,n)$ whenever $(h,m), (k,n) \in K \times N$ satisfy $k^{-1}h \in N$ and $hm = kn$. Let $f \in \cC(G)$ be the well-defined function $f(kn) = g(k,n)$, $k \in K, n \in N$. Then, $f \in \cD(G)$ if and only if $g(k,\cdot) \in \cD(N)$ for every $k \in K$ and $\set{g(\cdot,n) \in \cC(K)}{n \in N}$ is equicontinuous.
\end{thm}

In our case, $G = \R$, $N = \Z$, and $K = [0,1]$. We start with a real-valued function $h \in \cD(\Z)$ for which the averages $(1/N)\sum_{n = 0}^{N-1} h(n)$ diverge as $N$ grows. Such functions can be obtained from the system $(\T^2,T_a)$ for an arbitrary $a = \exx{\alpha} \in \T$ with $\alpha \in \R \setminus \Q$: since the system is minimal but not uniquely ergodic, there is a point $x \in \T^2$ and a function $k \in \cC(\T^2)$ such that the sequence of the averages $(1/N)\sum_{n = 0}^{N-1} k(T_a^n x)$ diverges as $N$ grows (see \cite{etvj}, Theorem 4.10, especially the alternative proof), so we may define $h$ as either the real or imaginary part of the function $n \mapsto k(T_a^n x) : \Z \to \C$, whichever has divergent averages. Moreover, we may assume that $h(2n) = 0$ for every $n \in \Z$ since, in general, $h = p_0 h + p_1 h$ where $p_i$ is the characteristic function of $2\Z + i$, $i = 0,1$; both $p_0 h$ and $p_1 h$ are distal, and at least one of them must have divergent averages, so we can replace $h$ by either $R_1 p_0 h$ or $p_1 h$ if necessary (recall that $R_1$ is the shift by $1$). We define $g \colon [0,1] \times \Z \to \C$ by $g(t,n) = (1-t)h(n) + th(n+1)$ for all $(t,n) \in [0,1] \times \Z$. This is a bounded, continuous function, and if $n \in \Z$, then $g(1,n) = g(0,n+1)$. Define $f \in \cC(\R)$ by $f(t+n) = g(t,n)$, $(t,n) \in [0,1] \times \Z$. Now, $g(t,\cdot) \in \cD(\Z)$ for any $t \in [0,1]$, and each $g(\cdot,n)$ is $\norm{h}$-Lipschitz, so $\set{g(\cdot,n)}{n \in \Z}$ is equicontinuous. The theorem above says that $f \in \cD(\R)$.

There are two strictly increasing sequences $(A_n)$ and $(B_n)$ in $\N$ and some distinct $a,b \in \R$ such that
\begin{equation*}
\frac{1}{A_n} \sum_{m = 0}^{A_n-1} h(m) \conv{n} a \hspace{15pt} \text{and} \hspace{15pt} \frac{1}{B_n} \sum_{m = 0}^{B_n-1} h(m) \conv{n} b.
\end{equation*}
We may modify the sequences $(A_n)$ and $(B_n)$ if necessary so that $A_n, B_n \in 2\N$ for all $n \in \N$. Define means $\alpha_n, \beta_n \in M(\cD(\R))$ for all $n \in \N$ by
\begin{equation*}
\alpha_n(q) = \frac{1}{A_n} \int_0^{A_n} q(x) \: dx \hspace{15pt} \text{and} \hspace{15pt} \beta_n(q) = \frac{1}{B_n} \int_0^{B_n} q(x) \: dx,
\end{equation*}
$q \in \cD(\R)$.
Let $\alpha, \beta \in \cD(\R)^{\ast}$ be weak$^{\ast}$ cluster points of the sequences $(\alpha_n)$ and $(\beta_n)$, respectively. They are invariant means (see \cite{bjm}, 3.4(d) on page 80), and it is easy to check that $\alpha(f) = a$ and $\beta(f) = b$. The $\R$-flow $F = (X_f,\sigma)$, where $X_f = \set{T_\mu f}{f \in \R^{\LC}}$ and $\sigma(t,f^{\prime}) = R_t f^{\prime}$ for all $t \in \R$ and $f^{\prime} \in X_f$, is minimal, distal, metric (the topologies of uniform convergence on compact sets and of pointwise convergence coincide on $X_f$ since $f$ is uniformly continuous) and admits at least two invariant measures. Consequently, there are at least two ergodic measures on $X_f$. To summarise, the following theorem is now proved:

\begin{thm}
\label{sec:cardr}
The cardinality of extreme invariant means on $\cD(\R)$, which is the cardinality of ergodic measures on the universal minimal distal $\R$-flow, is $2^{\cnt}$.
\end{thm}

Again, the cardinality of all invariant means on $\cD(\R)$ is likewise $2^{\cnt}$. Theorems~\ref{sec:cardz} and \ref{sec:cardr} also tell us something about ergodic measures on $\beta \Z$ and $\R^{\LC}$:

\begin{cor}
\label{sec:mlidc}
Let $X = \beta \Z$ or $X = \R^{\LC}$. For any minimal left ideal $M$ in $X$, there are $2^{\cnt}$ ergodic measures on $X$ whose support is contained in $M$.
\end{cor}

\begin{proof}
Let $T = \Z$ or $T = \R$ so that $X$ is a compactification of $T$. Let $M \subseteq X$ be a minimal left ideal, i.e., a minimal set for the flow $(T,X)$. Now, the subflow $(T,M)$ is a universal minimal $T$-flow, so there is a homomorphism $\pi \colon M \to T^{\cD}$. Any ergodic $\mu \in \cM(T,T^{\cD})$ can be lifted to an ergodic $\wt{\mu} \in \cM(T,M)$ with respect to $\pi$, that is, $\pi_{\ast} \wt{\mu} = \mu$. Thus, there are $2^{\cnt}$ ergodic measures on $M$. Each of them can be viewed as an invariant measure on $X$. Moreover, if $\wt{\mu}$ is ergodic on $M$, then it is also ergodic as an invariant measure on $X$.
\end{proof}

The final matter we address is the non-separability of the quotient space $\cD(T)/V$ for $T = \Z$ or $T = \R$ and for any closed subspace $V \subseteq \cD(T)$ such that every function in $V$ is almost convergent. Whichever the phase group, we can find an uncountable, multiply disjoint collection $\cX = \{X_i\}_{i \in I}$ of minimal distal $T$-flows, each admitting more than one invariant measure. For each $i \in I$, pick $x_i \in X_i$ and a real-valued function $g_i \in \cC(X_i)$ such that the function $f_i \in \cD(T)$ defined by $f_i(t) = g_i(tx_i)$, $t \in T$, is not almost convergent. The set
	\[K_i = \set{\mu(f_i) \in \R}{\mu \in IM(\cD(T))}
\]
is a non-degenerate closed interval for each $i \in I$. By scaling the functions $f_i$ and adding constants, if necessary, we may assume that $K_i = [0,1]$ for every $i \in I$. As argued in the proof of Theorem~\ref{sec:cardz}, the multiple disjointness of $\cX$ implies that, for any function $\omega \colon I \to [0,1]$, there exists some $\mu_\omega \in IM(\cD(T))$ such that $\mu_\omega(f_i) = \omega(i)$ for every $i \in I$ (the steps ensuring ergodicity are, of course, omitted here).

Next, define a function $R \colon \cD(T) \to \R$ by
	\[R(f) = \sup \set{|\mu(f) - \nu(f)|}{\mu,\nu \in IM(\cD(T))}
\]
for all $f \in \cD(T)$. Note that the supremum is attained for some means by the weak$^{\ast}$ compactness of $IM(\cD(T))$. The function $R$ has the following properties:
\begin{itemize}
\item[(i)] $R(f) = 0$ if and only if $f \in \cD(T)$ is almost convergent;
\item[(ii)] $R(f+g) \leq R(f) + R(g)$ for all $f,g \in \cD(T)$;
\item[(iii)] $R(f) \leq 2 \norm{f}$ for all $f \in \cD(T)$;
\item[(iv)] $R(f_i - f_j) \geq 2$ for all distinct $i,j \in I$.
\end{itemize}
The last estimate is obtained by choosing functions $\omega_1, \omega_2 \colon I \to [0,1]$ with $\omega_1(i) = \omega_2(j) = 1$ and $\omega_1(j) = \omega_2(i) = 0$, so $|\mu_{\omega_1}(f_i-f_j) - \mu_{\omega_2}(f_i-f_j)| = 2$.

Letting $B_i \subseteq \cD(T)$ denote the norm open ball centred at $f_i$ and of radius $1/2$ for each $i \in I$, we see that the collection $\{B_i + V\}_{i \in I}$ of non-empty open sets in the quotient space $\cD(T)/V$ is pairwise disjoint. For if $i,j \in I$ and if $(B_i + V) \cap (B_j + V) \neq \emptyset$, we can find $h_i \in B_i$ and $h_j \in B_j$ such that $h_i - h_j \in V$. We may write $h_i = g_i + f_i$ and $h_j = g_j + f_j$ for some $g_i,g_j \in \cD(T)$ such that $\norm{g_i}, \norm{g_j} < 1/2$. Then,
\begin{align*}
R(f_i - f_j) &\leq R(g_i - g_j) + R(h_i - h_j) \leq R(g_i) + R(g_j) \\
&\leq 2\norm{g_i} + 2\norm{g_j} < 2,
\end{align*}
so we must have $i = j$ by property (iv) in the list above. In conclusion, the space $\cD(T)/V$ has an uncountable, pairwise disjoint collection of non-empty open sets, completing the proof of the following theorem:

\begin{thm}
Let $T = \Z$ or $T = \R$, and let $V \subseteq \cD(T)$ be a closed subspace such that each function in $V$ is almost convergent. Then, the space $\cD(T)/V$ is not separable.
\end{thm}

For example, $V$ could be $\AP(T)$ or the space of all almost convergent functions in $\cD(T)$.

\section{Acknowledgements}

This work is part of my Ph.D. thesis prepared under the supervision of M. Filali at the University of Oulu. I would like to thank him for introducing me to the problems that inspired the research in this article and also for his helpful comments regarding the matters covered in the introduction.

\end{document}